\documentclass[12pt, twoside, leqno]{article}

\usepackage[utf8]{inputenc}
\usepackage{amsmath,amsthm}
\usepackage{amssymb}
\usepackage{enumerate}
\usepackage{graphicx}
\usepackage{tikz}
\usetikzlibrary{matrix,arrows}

\newtheorem{thm}{Theorem}[section]
\newtheorem{cor}[thm]{Corollary}
\newtheorem{lem}[thm]{Lemma}
\newtheorem{prop}[thm]{Proposition}
\newtheorem{prob}[thm]{Problem}

\theoremstyle{definition}
\newtheorem{defin}[thm]{Definition}

\newtheorem*{xrem}{Remark}

\newcommand{\zbp}{\emptyset}
\newcommand{\Z}{\mathbb{Z}}
\newcommand{\vare}{\varepsilon}
\newcommand{\varp}{\varphi}
\newcommand{\diam}{\operatorname{diam}}
\newcommand{\oL}{{\mathcal{L}}}
\newcommand{\oS}{{\mathcal{S}}}
\newcommand{\cont}{{\mathfrak{c}}}
\newcommand{\Spl}{\operatorname{Split}}
\newcommand{\non}{\operatorname{non}}
\newcommand{\ran}{\operatorname{ran}}
\newcommand{\NN}{\boldsymbol{\text{N}}}
\newcommand{\PN}{\boldsymbol{\text{PN}}}
\newcommand{\PNp}{\boldsymbol{\text{PN}\,'}}
\newcommand{\PMp}{\boldsymbol{\text{PM}\,'}}
\newcommand{\bPN}{\boldsymbol{\text{bPN}}}
\newcommand{\uPN}{\boldsymbol{\text{uPN}}}
\newcommand{\vPN}{\boldsymbol{\text{vPN}}}
\newcommand{\bPNp}{\boldsymbol{\text{bPN}\,'}}
\newcommand{\uPNp}{\boldsymbol{\text{uPN}\,'}}
\newcommand{\vPNp}{\boldsymbol{\text{vPN}\,'}}
\newcommand{\UN}{\boldsymbol{\text{UN}}}
\newcommand{\SN}{\boldsymbol{\text{SN}}}

\newcommand{\PM}{\boldsymbol{\text{PM}}}
\newcommand{\UM}{\boldsymbol{\text{UM}}}
\newcommand{\SM}{\boldsymbol{\text{SM}}}
\newcommand{\obc}{\mathord{\upharpoonright}}

\begin{document}


\baselineskip=17pt


\title{On the class of perfectly null sets and its transitive version}
\author{Michał Korch\\
Institute of Mathematics\\ 
University of Warsaw\\
02-097 Warszawa, Poland\\
E-mail: m\_{}korch@mimuw.edu.pl 
\and
Tomasz Weiss\\
Institute of Mathematics\\ 
Cardinal Stefan Wyszyński University in Warsaw\\
01-938 Warszawa, Poland\\
E-mail: tomaszweiss@o2.pl}

\date{}

\maketitle


\renewcommand{\thefootnote}{}



\renewcommand{\thefootnote}{\arabic{footnote}}
\setcounter{footnote}{0}


\bibliographystyle{siam}

\begin{abstract}
We introduce two new classes of special subsets of the real line: the class of perfectly null sets and the class of sets which are perfectly null in the transitive sense. These classes may play the role of duals to the corresponding classes on the category side. We investigate their properties and, in particular, we prove that every strongly null set is perfectly null in the transitive sense, and that it is consistent with ZFC that there exists a universally null set which is not perfectly null in the transitive sense. Finally, we state some open questions concerning the above classes. Although the main problem of whether the classes of perfectly null sets and universally null sets are consistently different remains open, we prove some results related to this question.
\end{abstract}

\section{Motivation and preliminaries}

Among classes of special subsets of the real line, the classes of perfectly meager sets (sets which are meager relative to any perfect set, here denoted by $\PM$) and universally null sets (sets which are null with respect to any possible finite diffused Borel measure, denoted by $\UN$) were considered to be dual (see \cite{am:ssrl}), though some differences between them have been observed. For example, the class of universally null sets is closed under taking products (see \cite{am:ssrl}), but it is consistent with ZFC that this is not the case for perfectly meager sets (see \cite{jp:ppmslf} and \cite{ir:ppms}). 

\begin{table}[h]
\centering
\begin{tabular}{c|ccccccc}
category & $\PM$ & $\supseteq$ & $\UM$ & $\supseteq$ & $\PMp$ & $\supseteq$ & $\SM$ \\ \hline
measure & \framebox{?} & & $\UN$ & & \framebox{?} & & $\SN$
\end{tabular}
\caption{Classes of special subsets of the real line.}
\label{classes}
\end{table}

In \cite{pz:ums}, P. Zakrzewski proved that two other earlier defined (see \cite{eg:afcs} and \cite{eg:spbfqcm}) classes of sets, and smaller then $\PM$, coincide and are dual to  $\UN$. Therefore, he proposed to call this class the universally meager sets (denoted by $\UM$). A set $A\subseteq 2^{\omega}$ is universally meager if every Borel isomorphic image of $A$ in $2^{\omega}$ is meager. The class $\PM$ was left without a counterpart (see Table~\ref{classes}), and in this paper (answering an oral question of P. Zakrzewski) we try to define a class of sets which may play the role of a dual class to $\PM$.

In the paper \cite{anmstw:assrnsmzs}, the authors introduced a notion of perfectly meager sets in the transitive sense (denoted here by $\PMp$), which turned out to be stronger than the classic notion of perfectly meager sets. In this article, we define an analogous class, which will be called the class of perfectly null sets in the transitive sense $\PNp$, and we investigate its properties. 

We assume that the reader is familiar with the standard terminology of special subsets of the reals, and we recall definitions that are less common (see also \cite{lb:srl}, \cite{tbhj:stsrl} and \cite{am:ssrl}).

Throughout this paper we will work generally in the Cantor space $2^{\omega}$. The basic clopen set in $2^{\omega}$ determined by a finite sequence $w\in 2^{<\omega}$ will be denoted by $[w]$. If $F$ is a set of partial functions $\omega\to 2$, the expression $[F]$ denotes $\bigcup_{f\in F}\{x\in 2^\omega\colon x\obc_{\ran f}=f\}$. The Cantor space will also be considered as a vector space over $\Z_{2}$. In particular, for $A,B\subseteq 2^{\omega}$, we let $A+B=\{t+s\colon t\in A, s\in B\}$.

Recall that a set $A$ is strongly null (strongly of measure zero) if for any sequence of positive $\delta_{n}>0$, there exists a sequence of open sets $\left<A_{n}\right>_{n\in\omega}$, with $\diam A_{n}<\vare_{n}$ for $n\in\omega$, where $\diam X$ denotes the diameter of a set $X$, and such that $A\subseteq\bigcup_{n\in\omega}A_{n}$. We denote the class of such sets by $\SN$. Galvin, Mycielski and Solovay (see \cite{fgjmrs:smzs}) proved that a set $A\in\SN$ (in $2^{\omega}$) if and only if for any meager set $B$, there exists $t\in 2^{\omega}$ such that $A\cap(B+t)=\zbp$. Therefore, one can consider a dual class of sets. A set $A$ is called strongly meager (strongly first category, denoted by $\SM$) if for any null set $B$, there exists $t\in 2^{\omega}$ such that $A\cap(B+t)=\zbp$. 

Finally, we shall say that an uncountable set $L\subseteq 2^{\omega}$ is a Lusin (respectively, Sierpiński) set if for any meager (respectively, null) set $X$, $L\cap X$ is countable.

\section{Perfectly null sets}

\subsection{Canonical measure on a perfect set}

If $P$ is a closed set in $2^{\omega}$, there is a pruned tree $T_{P}\subseteq 2^{<\omega}$ such that the set of all infinite branches of $T_{P}$ (usually denoted by $[T_{P}]$) equals $P$. If $T$ is a pruned tree, then $[T]$ is perfect if and only if for any $w\in T$, there exist $w',w''\in T$ such that $w\subseteq w', w\subseteq w''$, but $w'\nsubseteq w''$ and $w''\nsubseteq w'$. Such a tree is called a perfect tree.

If $w\in 2^n$, and $a,b\in\omega$, with $a\leq b$, then by $w[a,b]\in 2^{b-a+1}$ we denote a finite sequence such that $w[a,b](i)=w(a+i)$ for $i\leq b-a$. If $\left<s_0, s_1,\ldots s_k\right>$ is a finite sequence of natural numbers less than $n$, then $w\left<s_0, s_1,\ldots s_k\right>\in 2^{k+1}$ denotes a sequence such that $w\left<s_0, s_1,\ldots s_k\right>(i)=w(s_i)$ for any $i\leq k$. 

A finite sequence $w\in T_{P}$ will be called a branching point of a perfect set $P$ if $w^\frown 0,w^\frown 1\in T_{P}$. A branching point is on level $i\in\omega$ if there exist $i$ branching points below it. The set of all branching points of $P$ on level $i$ will be denoted by $\Spl_{i}(P)$ and $\Spl(P)=\bigcup_{i\in\omega} \Spl_{i}(P)$. Let $s_{i}(P)=\min\{|w|\colon w\in\Spl_{i}(P)\}$ and $S_{i}(P)=\max\{|w|\colon w\in\Spl_{i}(P)\}$. For $i>0$, we say that $w\in T_{P}$ is on level $i$ in $P$ (denoted by $l_P(w)=i$) if there exists $v,t\in T_P$ such that $v\subsetneq w\subseteq t$, $v\in\Spl_{i-1}(P), t\in\Spl_{i}(P)$. We say that $w\in T_{P}$ is on level $0$ if $w\subseteq t$ where $t\in \Spl_{0}(P)$.

Let $P$ be a perfect set in $2^{\omega}$ and $h_{P}\colon 2^{\omega}\to P$ be the homeomorphism given by the order isomorphism of $2^{<\omega}$ and $\Spl(P)$. We call this homeomorphism the canonical homeomorphism. Let $m$ denote the Lebesgue measure (the standard product measure) on $2^{\omega}$.

\begin{defin}
Let $A\subseteq P$ be such that $h_{P}^{-1}[A]$ is measurable in $2^{\omega}$. We define $\mu_{P}(A)=m(h_{P}^{-1}[A])$. 
Measure $\mu_{P}$ will be called \emph{the canonical measure on $P$}. A set $A\subseteq P$ such that $\mu_{P}(A)=0$ will be called $P$-null, a~set measurable with regard to $\mu_{P}$ will be called $P$-measurable. 
\end{defin}

The same idea of the canonical measure on a perfect set was used in \cite{mbam:mensmndcs}.

\begin{xrem}Sometimes measure $\mu_{P}$ will be considered as a measure on the whole $2^{\omega}$ by setting $\mu_{P}(A)=\mu_{P}(A\cap P)$ for $A\subseteq 2^{\omega}$ such that $A\cap P$ is $P$-measurable. 
\end{xrem}

For $w\in T_{P}$, we set $[w]_{P}=[w]\cap P$. Notice that if $w\in T_{P}$ is on level $i$ in $P$, then $\mu_{P}([w]_{P})=1/2^{i}$. If $Q\subseteq P$ is perfect, then $T_{Q}\subseteq T_{P}$, and therefore if $w\in T_{Q}$, then $l_Q(w)\leq l_P(w)$, so $\mu_{Q}([w]_{Q})\geq\mu_{P}([w]_{P})$. By defining the outer measure $\mu_{P}^{*}(A)=m^{*}(h_{P}^{-1}[A])$, where $m^{*}$ is the Lebesgue outer measure, we obtain the following proposition.

\begin{prop}
If $Q,P$ are perfect sets such that $Q\subseteq P$, and $A\subseteq Q$, then $\mu^{*}_{P}(A)\leq \mu^{*}_{Q}(A)$. In particular, every $Q$-null set $A\subseteq Q$ is also $P$-null.
\end{prop}
\hfill $\square$

\begin{prop}
If $Q,P$ are perfect sets such that $Q\subseteq P$, and $A$ is a $Q$-measurable subset of $Q$, then it is $P$-measurable.
\end{prop}
\begin{proof} If $A$ is $Q$-measurable, there exists a Borel set $B\subseteq 2^{\omega}$ such that $B\cap Q\subseteq A$ and $\mu_{Q}(A\setminus B)=0$, so $\mu_{P}(A\setminus B)=0$. Let $B'=B\cap Q$. $B'$ is Borel, $\mu_{P}(A\setminus B')=\mu_{P}(A\setminus B)=0$ and $B'\subseteq A$. \end{proof}

\begin{cor}\label{mp-podperfect}
If $P$ is perfect, and $Q_{n}\subseteq P$ for $n\in\omega$ are perfect sets such that $\mu_{P}(\bigcup_{n} Q_{n})=1$ and $A\subseteq P$ is such that for any $n\in\omega$, $A\cap Q_{n}$ is $Q_{n}$-measurable, then $A$ is $P$-measurable and $\mu_{P}(A)\leq\sum_{n\in\omega}\mu_{Q_{n}}(A\cap Q_{n})$. In particular, if for all $n\in\omega$, $A\cap Q_{n}$ is $Q_{n}$-null, then $A$ is $P$-null.
\end{cor}
\hfill $\square$

We will need the following lemma.

\begin{lem}\label{lem-measure-p-0}
Let $P\subseteq 2^\omega$ be a perfect set, $k\in \omega$ and $X\subseteq 2^{\omega}$ be such that for all $t\in P$, there exist infinitely many $n\in\omega$ such that there exists $w\in 2^k$ with $[t\obc_n\,^\frown w]_P\subseteq P\setminus X$. Then $\mu_P(X)=0$.
\end{lem}

\begin{proof}
Notice that if $k=0$, then $X\cap P=\zbp$, so we can assume that $k>0$. We prove by induction that for any $m\in\omega$, there exists a finite set $S_{m}\subseteq T_P$ such that $X\cap P\subseteq \bigcup_{s\in S_m}[s]_P$, and 
\[\sum_{s\in S_m}\dfrac{1}{2^{l_P(s)}}\leq \left(\dfrac{2^k-1}{2^k}\right)^m.\]

Let $S_0=\{\zbp\}$. Given $S_m$, for each $s\in S_m$ and each $t\in P$ such that $s\subseteq t$, we can find $s_{s,t}\in T_P$ such that $s\subseteq s_{s,t}\subseteq t$ and $w_{s,t}\in 2^k$ with $[s_{s,t}\,^\frown w_{s,t}]_P\subseteq P\setminus X$. Therefore, since $[s]_P$ is compact, we can find a finite set $A_s\subseteq P$ such that $[s]_P=\bigcup_{t\in A_s}[s_{s,t}]_P$ and $[s_{s,t}]_P\cap [s_{s,t'}]_P=\zbp$ if $t,t'\in A_s$ and $t\neq t'$. Let 
\[S_{m+1}=\left\{s_{s,t}\,^\frown w\colon s\in S_m\land t\in A_s\land w\in 2^k\setminus \{w_{s,t}\}\right\}\cap T_P.\]

We have that $X\cap P\subseteq \bigcup_{s\in S_{m+1}}[s]_P$. Notice also that for $s\in S_m$, 
 \[\sum_{t\in A_s} \frac{1}{2^{l_P(s_{s,t})}}=\frac{1}{2^{l_P(s)}}.\]

  Moreover, if $t\in A_s$, then 
  \[\sum_{w\in 2^k\setminus \{w_{s,t}\}} \frac{1}{2^{l_P(s_{s,t}\,^\frown w)}}\leq \frac{2^k-1}{2^k}\cdot\frac{1}{2^{l_P(s_{s,t})}}.\]
   Therefore,
  \[\sum_{s\in S_{m+1}}\frac{1}{2^{l_P(s)}}\leq \dfrac{2^k-1}{2^k}\cdot\sum_{s\in S_m}\frac{1}{2^{l_P(s)}}\leq \left(\dfrac{2^k-1}{2^k}\right)^{m+1},\]

   which concludes the induction argument.

Thus, $\mu_{P}(X)\leq \left(1-1/2^k\right)^m$ for any $m\in\omega$, and so $\mu_P(X)=0$.
\end{proof}

Now, we define a possible measure analogue of the class of perfectly meager sets. 

\begin{defin}
We shall say that $A\subseteq 2^{\omega}$ is \emph{perfectly null} if it is null in any perfect set $P\subseteq 2^{\omega}$ with respect to measure $\mu_{P}$. The class of perfectly null sets will be denoted by $\PN$.
\end{defin}

\begin{prop}\label{prop-pn}
The following conditions are equivalent for a set $A\subseteq 2^{\omega}$:
\begin{enumerate}[\upshape (i)]
\item $A$ is perfectly null,
\item for every perfect $P\subseteq 2^{\omega}$, $A\cap P$ is $P$-measurable, but $P\setminus A\neq\zbp$,
\item there exists $n\in\omega$ such that for every $w\in 2^{n}$ and every perfect $P\subseteq [w]$, $A\cap P$ is $P$-null.
\end{enumerate}
\end{prop}
\begin{proof} Notice that if $A\cap P$ is $P$-measurable with $\mu_{P}(A\cap P)>0$, then we can find a closed uncountable set $F$ such that $F\subseteq A\cap P$. Therefore, there is a perfect set $Q\subseteq F$ and $Q\subseteq A$, so $Q\setminus A=\zbp$. Moreover, given any perfect set $P$ we have $P=\bigcup_{w\in 2^{n},w\in T_{P}} [w]_{P}$, and for any $w\in 2^{n}$ such that $w\in T_{P}$, the set $[w]_{P}$ is perfect. \end{proof}

\vspace{0.1cm}

\subsection{The main open problem}  

\begin{prop}\label{un-is-pn}
$\UN\subseteq \PN$.
\end{prop}

\begin{proof} Let $A\subseteq 2^{\omega}$ be universally null, and let $P$ be perfect. Let $\lambda$ be a~measure on $2^{\omega}$ such that $\lambda(B)=\mu_{P}(B\cap P)$ for any Borel set $B\subseteq 2^\omega$. Then $\lambda(A)=0$, so $A$ is $P$-null. \end{proof}

Unfortunately, we do not know the answer to the following question.

\begin{prob} \label{cru-pro}
Is it consistent with ZFC that $\UN\neq\PN$?
\end{prob}

\begin{xrem}
On the category side every proof of the consistency of the fact that $\UM\neq \PM$ known to the authors uses the idea of the Lusin function or similar arguments. The Lusin function $\oL\colon \omega^{\omega}\to 2^{\omega}$ was defined in \cite{nl:etpc}, and extensively described in \cite{ws:hc}. To get the Lusin function we construct a system $\left<P_{s}\colon s\in \omega^{<\omega}\right>$ of perfect sets, such that for $s\in\omega^{<\omega}$ and $n,m\in\omega$, $\diam P_{s}\leq 1/2^{|s|}$, $P_{s^\frown n}\subseteq P_{s}$ is nowhere dense in $P_{s}$, $\bigcup_{k\in\omega}P_{s^\frown k}$ is dense in $P_{s}$, and if $n\neq m$, then $P_{s^\frown n}\cap P_{s^\frown m}=\zbp$. Next, we set $\oL(x)$ to be the only point of $\bigcap_{n\in\omega} P_{x|_n}$. One can prove that $\oL$ is a continuous and one-to-one function. Furthermore, if $Q\subseteq 2^{\omega}$ is a perfect set, then $\oL^{-1}[\bigcup\{P_{s}\colon P_{s}\text{ is nowhere dense in }Q\}]$ contains an open dense set. Therefore, if $L$ is a Lusin set, then $\oL[L]$ is perfectly meager (see also \cite{am:ssrl}). Moreover, $\oL^{-1}$ is a function of the first Baire class. Given such a function it easy to see that if there exists a Lusin set $L$, then $\UM\neq\PM$. This should be clear since $\UM$ is a class closed under taking Borel isomorphic images, so $\oL[L]\in \PM\setminus\UM$.
\end{xrem}

Therefore, to prove $\PN\neq \UN$, we possibly need some analogue of the Lusin function. 

\begin{prob}
Is there an analogue of the Lusin function for perfectly null sets?
\end{prob}

But even if such an analogue exists, it cannot be constructed in a similar way to the Lusin's argument. Indeed, if we equip $\omega^\omega$ with the natural measure $m$ which is defined by the following formula
\[m([w])=\prod_{i=0}^{|w|-1}\frac{1}{2^{w(i)+1}},\]
where $w\in\omega^{<\omega}$ and $[w]=\{f\in\omega^\omega\colon w\subseteq f\}$, we get the following proposition.

\begin{prop}
Let $\oS\colon \omega^\omega\to 2^\omega$ be a function such that there exists a sequence $\left<P_s\colon s\in \omega^{<\omega}\right>$  such that for $s\in \omega^{<\omega}$, $P_{s}\subseteq 2^\omega$ is a perfect set, and for $n,m\in\omega$, $n\neq m \rightarrow P_{s^\frown n}\cap P_{s^\frown m}=\zbp$, $P_{s^\frown n}\subseteq P_{s}$,  $\diam P_{s}\leq 1/2^{|s|}$, and  $\oS(x)$ is the only element of $\bigcap_{n\in\omega} P_{x\obc_n}$. Then there exists a perfect set $Q\subseteq 2^\omega$ such that $m(\oS^{-1}[\bigcup\{P_{s}\colon \mu_{Q}(P_{s})=0\}])<1$.
\end{prop} 

\begin{proof} We define $T\subseteq \omega^{<\omega}$ inductively as follows: in the $n$-th step we  construct $T_{n}=T\cap \omega^n$, such that $|T_n|<\omega$ for all $n\in\omega$. Let $T_{0}=\{\zbp\}$. Assume that $T_{n}$ is constructed and $w\in T_n$. Let $M_{w} \geq 2$ be such that $2^{M_w}\geq 2^{n+2}\cdot |T_{n}|\cdot m([w])$ and $T_{n+1}=\{w^\frown k\colon w\in T_{n}\land k\in\omega\land k<M_{w}\}$. 

Therefore, if $w\in T_n$, then
\begin{equation*}
\begin{split}
 m\left([w]\setminus \bigcup\{w^\frown k\colon k< M_w\}\right)=m\left(\bigcup\{w^\frown k\colon k\geq M_w\}\right)=\\
=m([w])\cdot \sum_{i=M_w}^{\infty}\frac{1}{2^{i+1}}=\frac{m([w])}{2^{M_w}}\leq\frac{1}{2^{n+2}|T_n|}. 
\end{split}
\end{equation*}
Thus, for all $n\in\omega$,
\[m\left(\bigcup\{[s]\colon s\in T_{n}\}\setminus \bigcup\{[s]\colon s\in T_{n+1}\}\right)\leq\frac{1}{2^{n+2}},\]
so
\[m\left(\bigcup\{[s]\colon s\notin T\}\right)=m\left(\bigcup_{n\in\omega}\left(\bigcup\{[s]\colon s\in T_{n}\}\setminus \bigcup\{[s]\colon s\in T_{n+1}\right)\right)\leq\frac{1}{2}.\]
 Let $Q=\bigcap_{n\in\omega}\bigcup_{s\in T_{n}} P_{s}$. Obviously, $Q$ is a closed set. Moreover, if $s\in T$, there exists $w\in 2^{<\omega}$, $[w]_{Q}\subseteq P_{s}$. It should be clear since for all $n\in\omega$, $\{P_{s}\colon s\in T_{n}\}$ is a finite collection of disjoint perfect sets, and $Q\subseteq \bigcup_{s\in T_{n}} P_s$. Therefore, $Q$ is perfect and $\mu_{Q}(P_{s})>0$. On the other hand, if  $s\notin T$, then $P_{s}\cap Q=\zbp$, so $\mu_{Q}(P_{s})=0$. Therefore, if $\oS(x)\in P_{s}$ and $\mu_{Q}(P_{s})=0$, then $s\notin T$ and $x\in [s]$, so $m(\oS^{-1}[\bigcup\{P_{s}\colon \mu_{Q}(P_{s})=0\}])= m(\bigcup\{[s]\colon s\notin T\})\leq 1/2$. 
\end{proof}

Obviously, since for every diffused Borel measure $\mu$, there exists a Borel isomorphism of $2^{\omega}$ mapping $\mu$ to the Lebesgue measure (see e.g. \cite[Theorem 4.1(ii)]{em:zfbm}), if the class $\PN$ is closed under Borel automorphisms of $2^\omega$, then $\UN=\PN$, which motivates the following question, which was asked by the reviewer.

\begin{prob}
Is the class $\PN$ closed under homeomorphisms of $2^\omega$ onto itself?
\end{prob}

\subsection{Simple perfect sets}

To understand what may happen in the solution of the main open problem which was mentioned above, we restrict our attention to some special subfamilies of all perfect sets. This will lead to an important result in Theorem \ref{upnprod}. 

\begin{defin}
A perfect set $P$ will be called \emph{balanced} if $s_{i+1}(P)>S_{i}(P)$ for all $i\in\omega$. This definition  generalizes the notion of uniformly perfect set, which can be found in \cite{jbplst:rapspd}. A perfect set $P$ is \emph{uniformly} perfect if for any $i\in \omega$, either $2^{i}\cap T_{P}\subseteq \Spl(P)$ or $2^{i}\cap \Spl(P)=\zbp$. If additionally, in a uniformly perfect set $P$, $\forall_{w,v\in T_{P}, |w|=|v|}\forall_{j\in\{0,1\}}(w^\frown j\in T_P\rightarrow  v^\frown j\in T_P)$, then $P$ is called a \emph{Silver} perfect set (see for example \cite{mkantw:ssrtfn}). 
\end{defin}

A set that is null in any balanced (respectively, uniformly, Silver) perfect set will be called balanced perfectly null (respectively, uniformly perfectly null, Silver perfectly null). The class of such sets will be denoted by $\bPN$ (respectively, $\uPN$, $\vPN$). Obviously, $\PN\subseteq\bPN\subseteq\uPN\subseteq\vPN$. 

\begin{lem} \label{lem-bp}
There exists a perfect set $E$ such that for every balanced perfect set $B$, we have either $\mu_{B}(E)=0$ or $\mu_{E}(B)=0$.
\end{lem}
\begin{proof} Consider $K=\{000,001,011,111\}\subseteq 2^{3}$ and a perfect set $E\in 2^{\omega}$ such that $x\in E$ if and only if $x[3k,3k+2]\in K$ for every $k\in\omega$ (see Figure~\ref{rys-lem-bp}). Let $B$ be a balanced perfect set. Imagine now how $T_{B}$ looks like in a $K$-block of $T_{E}$ (see Figure~\ref{rys-lem-bp}, where $T_B$ is shown as doted lines). Let $k\in\omega$ and $w\in T_E\cap 2^{3k}$. The following two situations are possible. Either $\{w^\frown s\colon s\in K\}\subseteq T_{B}$ (possibility $(a)$), or alternatively $\{w^\frown s\colon s\in K\}\setminus T_{B}\neq\zbp$ (possibility $(b)$).  

\begin{figure}[h]
\centering
\includegraphics[scale=0.6]{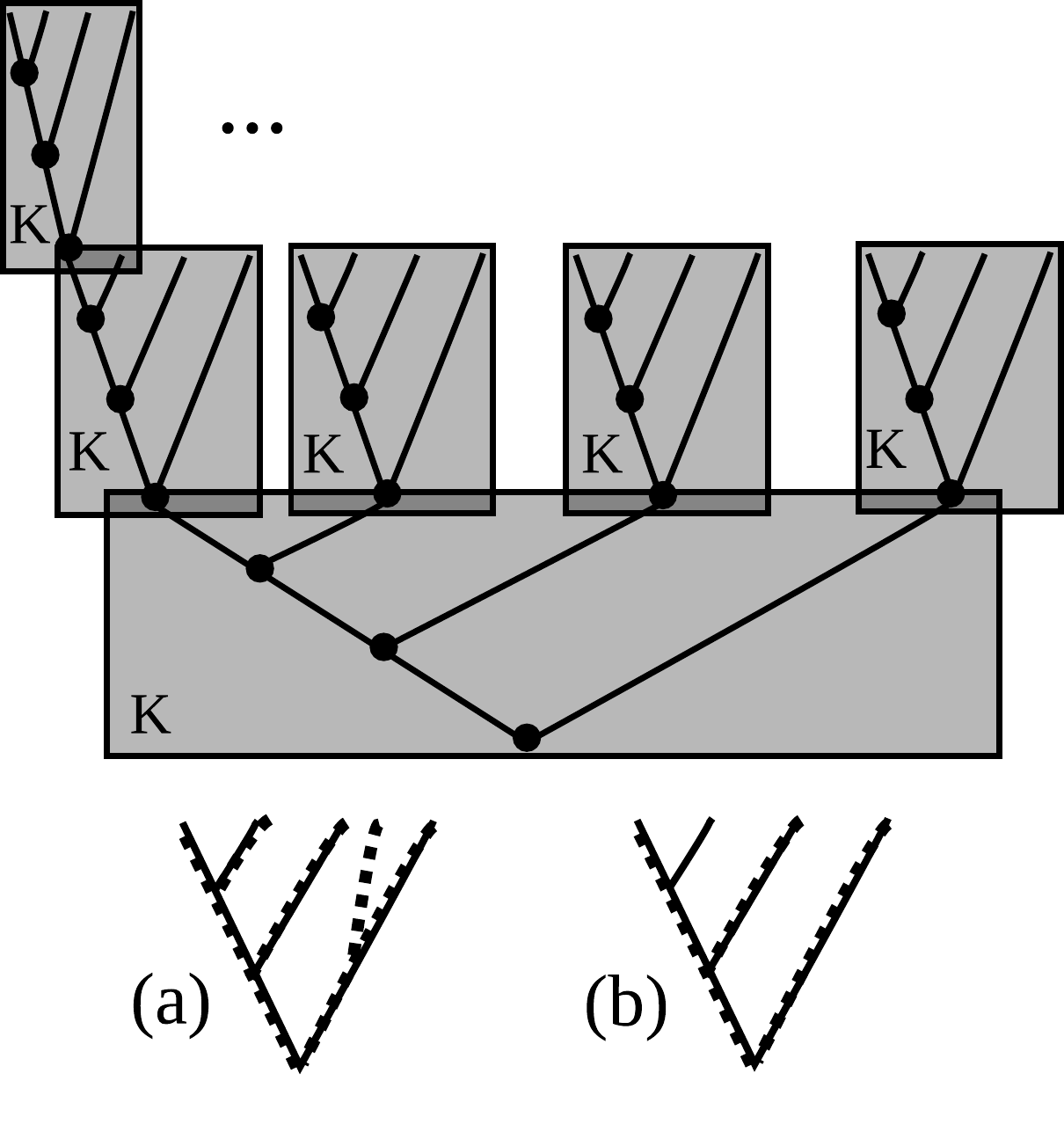} 
\caption{Proof of Lemma~\ref{lem-bp}.}
\label{rys-lem-bp}
\end{figure}

Assume that for all $t\in E$, there exist infinitely many $k\in\omega$ such that $\{t\obc_{3k}\,^\frown s\colon s\in K\}\setminus T_{B}\neq\zbp$ (case $(b)$). Then, by Lemma \ref{lem-measure-p-0}, $\mu_{E}(B)=0$. On the other hand, assume that there exists $t\in E$ such for all but finite $k\in\omega$, we have $\{t\obc_{3k}\,^\frown s\colon s\in K\}\subseteq T_{B}$ (case $(a)$). It follows that there exists $i\in\omega$, such that $B$ has a branching point of length $j$ for all $j\geq i$, so $s_{j+1}(B)\leq S_{j}(B)+1$, for any $j\geq i$. And since $B$ is a balanced perfect set, it implies, that $s_{j}(B)=S_j(B)$ and $s_{j+1}(B)=s_j(B)+1$ for any $j>i$. In other words, for $w\in T_{B}\cap 2^{i}$, $B\cap [w]=[w]$, and therefore for any $v\in T_B\cap 2^{3k}$ with $3k>i$, there exists $w\in 2^3$, such that $v^\frown w\in T_B\setminus T_E$. It follows that $\mu_{B}(E)=0$, by Lemma~\ref{lem-measure-p-0}. \end{proof}

\begin{prop}
Suppose that there exists a Sierpiński set. Then $\PN\subsetneq\bPN$.
\end{prop}

\begin{proof} Let $E$ be the perfect set defined in Lemma~\ref{lem-bp}, and let $S\subseteq E$ be a Sierpiński set with respect to $\mu_{E}$. Obviously, $S$ is not perfectly null. But if $B$ is a balanced perfect set, then either $\mu_{B}(E)=0$, so $\mu_{B}(S)=0$, or $\mu_{E}(B)=0$, so $S\cap B$ is countable. Thus, $\mu_{B}(S)=0$. So $S\in\bPN\setminus \PN$. \end{proof}

\begin{prop} \label{bpnupn}
$\bPN\subsetneq \uPN\subsetneq \vPN$.
\end{prop}
\begin{proof} 
The first inclusion is proper, because if we take any balanced perfect set $B$ such that for each $i\in\omega$, we have $|\Spl(B)\cap 2^{i}|=1$ and any uniformly perfect set $U$, then $\mu_{U}(B)\leq (n+1)/2^{n}$ for any $n\in\omega$, so $B$ is $U$-null. Thus, $B\in uPN\setminus \bPN$.

To see that the second inclusion is proper, notice that the uniformly perfect set $U=\{\alpha\in 2^{\omega}\colon\forall_{i\in\omega} \alpha(2i+1)=\alpha(2i)\}$ is null in every Silver perfect set. Indeed, let $S$ be a Silver perfect set. Let $i\in \omega$ be such that for every $w\in 2^{2i}\cap S$, $w\in \Spl(S)$, or for every $w\in 2^{2i+1}\cap S$, $w\in \Spl(S)$. The following two cases are possible:
\begin{itemize}
\item for every $w\in 2^{2i}\cap S$, $w\in \Spl(S)$, so $w^\frown 0, w^\frown 1\in T_S$. Then $w^\frown 0^\frown 1\in T_S$ or $w^\frown 0^\frown 0\in T_S$. In the first case $w^\frown 0^\frown 1\in T_S\setminus T_U$. In the second $w^\frown 1^\frown 0\in T_S$, but $w^\frown 1^\frown 0\notin T_U$.
\item for every $w\in 2^{2i}\cap S$, $w\notin \Spl(S)$. Without a loss of generality, assume that $w^\frown 0\in T_S$. Then $w^\frown 0\in \Spl(S)$, and $w^\frown 0^\frown 1\in T_S\setminus T_U$. 
\end{itemize}
Since there exist infinitely many $i\in\omega$ such that $2^{2i}\cap S\subseteq \Spl(S)$ or $2^{2i+1}\cap S\subseteq \Spl(S)$, Lemma \ref{lem-measure-p-0} can be applied to get that $\mu_S(U)=0$.
 \end{proof}

\begin{prop}
The following conditions are equivalent for a set $A\subseteq 2^{\omega}$:
\begin{enumerate}[\upshape (i)]
\item $A$ is perfectly null,
\item for every perfect set  $P\subseteq 2^{\omega}$, $A\cap P$ is $P$-measurable, but for every balanced perfect set $Q\subseteq 2^{\omega}$, $Q\setminus A\neq\zbp$
\item for every perfect set  $P\subseteq 2^{\omega}$, $A\cap P$ is $P$-measurable and $A\in \bPN$.
\end{enumerate}
\end{prop}
\begin{proof} Notice that there exists a balanced perfect set in every perfect set. Therefore, in the proof of Proposition~\ref{prop-pn} we can require that the perfect set $Q$ is balanced. \end{proof}

\begin{xrem}
Notice that, even if a set is $P$-measurable for any perfect set and does not contain any uniformly perfect set, it needs not to be perfectly null. An example of such a set is the set $B$ from the proof of Proposition~\ref{bpnupn}. 
\end{xrem}

\begin{prop}
\begin{enumerate}[\upshape (i)]
\item $A\in\bPN$ if and only if for every balanced perfect $P\subseteq 2^{\omega}$, $A\cap P$ is $P$-measurable, but $P\setminus A\neq\zbp$.
\item $A\in\uPN$ if and only if for every uniformly perfect $P\subseteq 2^{\omega}$, $A\cap P$ is $P$-measurable, but $P\setminus A\neq\zbp$.
\item $A\in\vPN$ if and only if for every Silver perfect $P\subseteq 2^{\omega}$, $A\cap P$ is $P$-measurable, but $P\setminus A\neq\zbp$.
\end{enumerate}
\end{prop}

\begin{proof} We proceed as in the proof of Proposition~\ref{prop-pn}. For uniformly and Silver perfect sets we use  {\cite[Lemma~2.4]{mkantw:ssrtfn}}, which states that there exists a~Silver perfect set in every set of positive Lebesgue measure, and we notice that if $P$ is a uniformly (respectively, Silver) perfect set, and $h_{P}\colon 2^{\omega}\to P$ is the canonical homeomorphism, then the image of any Silver perfect set is uniformly (respectively, Silver) perfect. \end{proof}

\subsection{Perfectly null sets and $s_{0}$ and $v_{0}$ ideals}

Recall that a set $A$ is a Marczewski $s_{0}$-set if for any perfect set $P$, there exists a perfect set $Q\subseteq P$ such that $Q\cap A=\zbp$.

\begin{prop}
$\PN\subseteq\bPN\subseteq s_{0}$.
\end{prop}

\begin{proof} Indeed, if $P$ is perfect and $X\in \bPN$, let $B\subseteq P$ be a balanced perfect set. Then $\mu_{B}(B\setminus X)=1$, so there exists a closed set $F\subseteq B\setminus X$ of positive measure. Therefore, it is uncountable, and there exists a perfect set $Q\subseteq F\subseteq P\setminus X$. \end{proof}

\begin{xrem}Obviously, $\uPN\not\subseteq s_{0}$ (see the proof of Proposition \ref{bpnupn}).\end{xrem}

We say that a set $X$ has $v_{0}$ property if for every Silver perfect set $P$, there exists a Silver perfect set $Q\subseteq P\setminus X$ (see \cite{mkantw:ssrtfn}). 

\begin{prop}
$\PN\subseteq\vPN\subseteq v_{0}$. 
\end{prop}

\begin{proof} Let $P\subseteq 2^{\omega}$ be a Silver perfect set, and let $X\in\vPN$. Notice that the image of any Silver perfect set under the canonical homeomorphism $h_{P}\colon 2^{\omega}\to P$ is a Silver perfect set. Since $m(2^{\omega}\setminus h_{P}^{-1}[X])=1$, there exists a~Silver perfect set $Q\subseteq 2^{\omega}\setminus h_{P}^{-1}[X]$ (see \cite[Lemma 2.4]{mkantw:ssrtfn}). So, $h_{P}[Q]\subseteq P\setminus X$ is a Silver perfect set. \end{proof}  

M.~Scheepers (see \cite{ms:apsrnig}) proved that if $X$ is a measure zero set with $s_0$ property, and $S$ is a Sierpiński set, then $X+S$ is also an $s_{0}$-set. Therefore, we easily obtain the following proposition.

\begin{prop}\label{pn-sierp-s0}
The algebraic sum of a Sierpiński set and a perfectly null set is an $s_{0}$-set.
\end{prop}

\hfill $\square$

\subsection{Products}

We consider $\PN$ sets in the product $2^{\omega}\times 2^{\omega}$ using the natural homeomorphism $h\colon 2^{\omega}\times 2^{\omega}\to 2^{\omega}$ defined as $h(x,y)=\left<x(0),y(0),x(1),y(1),\ldots\right>$.

It is consistent with ZFC that the product of two perfectly meager sets is not perfectly meager (see \cite{ir:ppms}, \cite{jp:ppmslf}).  If the answer to the Problem~\ref{cru-pro} is positive, then it makes sense to ask the following question.

\begin{prob} \label{pn-prod}
Is the product of any two perfectly null sets perfectly null?
\end{prob}

This problem remains open, but in the easier case of Silver perfect sets, the answer is in the affirmative. First, notice the following simple lemma.

\begin{lem}\label{perf-prod}
Let $P,Q\subseteq 2^{\omega}$ be perfect sets. Then $\mu_{P\times Q}=\mu_P\times \mu_Q$. In particular, if $X\subseteq 2^{\omega}\times 2^{\omega}$ is such that $\pi_{1}[X]$ is $P-null$, then $\mu_{P\times Q}(X)=0$. 
\end{lem}
\begin{proof} 
First, we shall prove that for any $n\in\omega$ and any $v\in 2^{2n}$,
\[\mu_{P\times Q}\left([v]_{P\times Q}\right)=\frac{1}{2^{l_P(w_P)}}\cdot \frac{1}{2^{l_Q(w_Q)}},\]
where $w_P,w_Q\in 2^n$ are such that for any $i<n$, $w_P(i)=v(2i)$ and $w_Q(i)=v(2i+1)$. This assertion can be proved by induction on $n$. For $n=0$, we get $v=w_P=w_Q=\zbp$, and $\mu_{P\times Q}([v]_{P\times Q})=1=\frac{1}{2^{l_P(w_P)}}\cdot \frac{1}{2^{l_Q(w_Q)}}$.

Now consider $v\in 2^{2(n+1)}$. Then 
\begin{itemize}
\item if both $w_P\obc_n$ and $w_Q\obc_n$ are branching points in $P$ and $Q$ respectively (so $l_P(w_P)=l_P(w_P\obc_n)+1$ and $l_Q(w_Q)=l_Q(w_Q\obc_n)+1$), then $v\obc_{2n}\in \Spl(P\times Q)$ and $v\obc_{2n+1}\in \Spl(P\times Q)$, and so $\mu_{P\times Q}([v]_{P\times Q})=1/2\cdot 1/2\cdot \mu_{P\times Q}([v\obc_{2n}]_{P\times Q})= 1/2\cdot 1/2 \cdot 1/2^{l_P(w_P\obc_n)}\cdot 1/2^{l_Q(w_Q\obc_n)}=1/2^{l_P(w_P)}\cdot 1/2^{l_Q(w_Q)}$. 
\item if $w_P\obc_n$ or $w_Q\obc_n$, but not both, is a branching point in $P$ or $Q$ respectively, we may assume without a loss of generality that $w_P\obc_n\in \Spl(P)$ and $w_Q\obc_n\notin \Spl(Q)$ (so $l_P(w_P)=l_P(w_P\obc_n)+1$ and $l_Q(w_Q)=l_Q(w_Q\obc_n)$). Then $v\obc_{2n}\in \Spl(P\times Q)$, but $v\obc_{2n+1}\notin \Spl(P\times Q)$, and so $\mu_{P\times Q}([v]_{P\times Q})=1/2\cdot 1\cdot \mu_{P\times Q}([v\obc_{2n}]_{P\times Q})= 1/2\cdot 1 \cdot 1/2^{l_P(w_P\obc_n)}\cdot 1/2^{l_Q(w_Q\obc_n)}=1/2^{l_P(w_P)}\cdot 1/2^{l_Q(w_Q)}$. 
\item if $w_P\obc_n\notin \Spl(P)$ and $w_Q\obc_n\notin \Spl(Q)$ (so $l_P(w_P)=l_P(w_P\obc_n)$ and $l_Q(w_Q)=l_Q(w_Q\obc_n)$), then $v\obc_{2n}, v\obc_{2n+1}\notin \Spl(P\times Q)$, and so $\mu_{P\times Q}([v]_{P\times Q})=1\cdot 1\cdot \mu_{P\times Q}([v\obc_{2n}]_{P\times Q})= 1/2^{l_P(w_P\obc_n)}\cdot 1/2^{l_Q(w_Q\obc_n)}=1/2^{l_P(w_P)}\cdot 1/2^{l_Q(w_Q)}$, 
\end{itemize}
which concludes the induction argument. Since every open set in $P\times Q$ is a countable union of sets of form $[v]_{P\times Q}$, with $v\in 2^{2n}$, $n\in\omega$, this concludes the proof of the Lemma.

\end{proof}

\begin{prop}\label{vpnprod}
If $X,Y\in\vPN$, then $X\times Y\in \vPN$ in $2^{\omega}\times 2^{\omega}$.
\end{prop}

\begin{proof} Fix a Silver perfect set $P$. Recall that such a set is uniquely defined by a sequence $\left<a_n\right>_{n\in\omega}$, $a_n\in\{-1,0,1\}$ such that $\{n\in\omega\colon a_n=-1\}$ is infinite, $T_P$ splits on all branches at length $n\in \omega$ if and only if $a_n=-1$, and $t(n)=a_n$ for all $t\in P$ for any other $n\in\omega$. Let $T_1$ be a tree which splits on all branches at length $n$ if and only if $a_{2n}=-1$, and $t(n)=a_{2n}$ for any $t\in[T_1]$ for any other $n\in\omega$. Finally, let $T_2$ be a tree which splits on all branches at length $n$ if and only if $a_{2n+1}=-1$, and $t(n)=a_{2n+1}$ for any $t\in[T_2]$ for any other $n\in\omega$. Let $P_1=[T_1]$ and $P_2=[T_2]$. If $\{2n\in\omega\colon a_n=-1\}$ is infinite, then $P_1$ is a Silver perfect set. On the other hand, if $\{2n\in\omega\colon a_n=-1\}$ is finite, then $P_1$ is also finite.  Analogously, if $\{2n+1\in\omega\colon a_n=-1\}$ is infinite, then $P_2$ is a Silver perfect set. On the other hand, if  $\{2n+1\in\omega\colon a_n=-1\}$ is finite, then $P_2$ is also finite. Moreover, $P=P_1\times P_2$.

If $P_{1}$ and $P_{2}$ are Silver perfect, then by Lemma~\ref{perf-prod}, $\mu_{P}(X\times Y)=0$. 

The other case is when $P_1$ or $P_2$, but not both, is finite. Without a loss of generality, we may assume that $P_2$ is finite. Then $P=\bigcup_{t\in P_2} P_1\times\{t\}$. Obviously, for any $t\in Y$, $\mu_{P_1\times \{t\}}(X\times Y)=\mu_{P_1}(X)=0$, so by Corollary~\ref{mp-podperfect}, also $\mu_{P}(X\times Y)=0$.\end{proof}

On the other hand, it is consistent with ZFC that the classes $\uPN$ and $\bPN$ are not closed under taking products.

\begin{thm}
\label{upnprod}
If there exists a Sierpiński set, then there are $X,Y\in\bPN$, such that $X\times Y\notin\uPN$.
\end{thm}

\begin{proof}Let $J\subseteq 2^{8}$ be as shown in Figure~\ref{rys-prod-upn} ($J=\{00000000,$ $00010111,$ $00101011,$ $00111111,$ $01001010,$ $01011111,$ $01101011,$ $01111111,$ $10000101,$ $10010111,$ $10101111,$ $10111111,$ $11001111,$ $11011111,$ $11101111,$ $11111111\}$). Let $P$ be a perfect set such that $x\in P$ if and only if for all $n\in\omega$, $x[8n,8n+7]\in J$. Obviously, $P$ is a uniformly perfect set. Let $Q=\pi_{1}[P]$. Notice that $x\in Q$ if and only if for all $n\in\omega$, $x[4n,4n+3]\in L$, where $L=\{0000,$ $0001,$ $0011,$ $0111,$ $1000,$ $1001,$ $1011,$ $1111\}\subseteq 2^4$ (see Figure~\ref{rys-prod-upn} and Table~\ref{t1-upnprod}).

  Notice that $L$ consists of two $K$-blocks (see the proof of Lemma~\ref{lem-bp}) joined by an additional root.

\begin{figure}[h]
\centering
\includegraphics[scale=0.6]{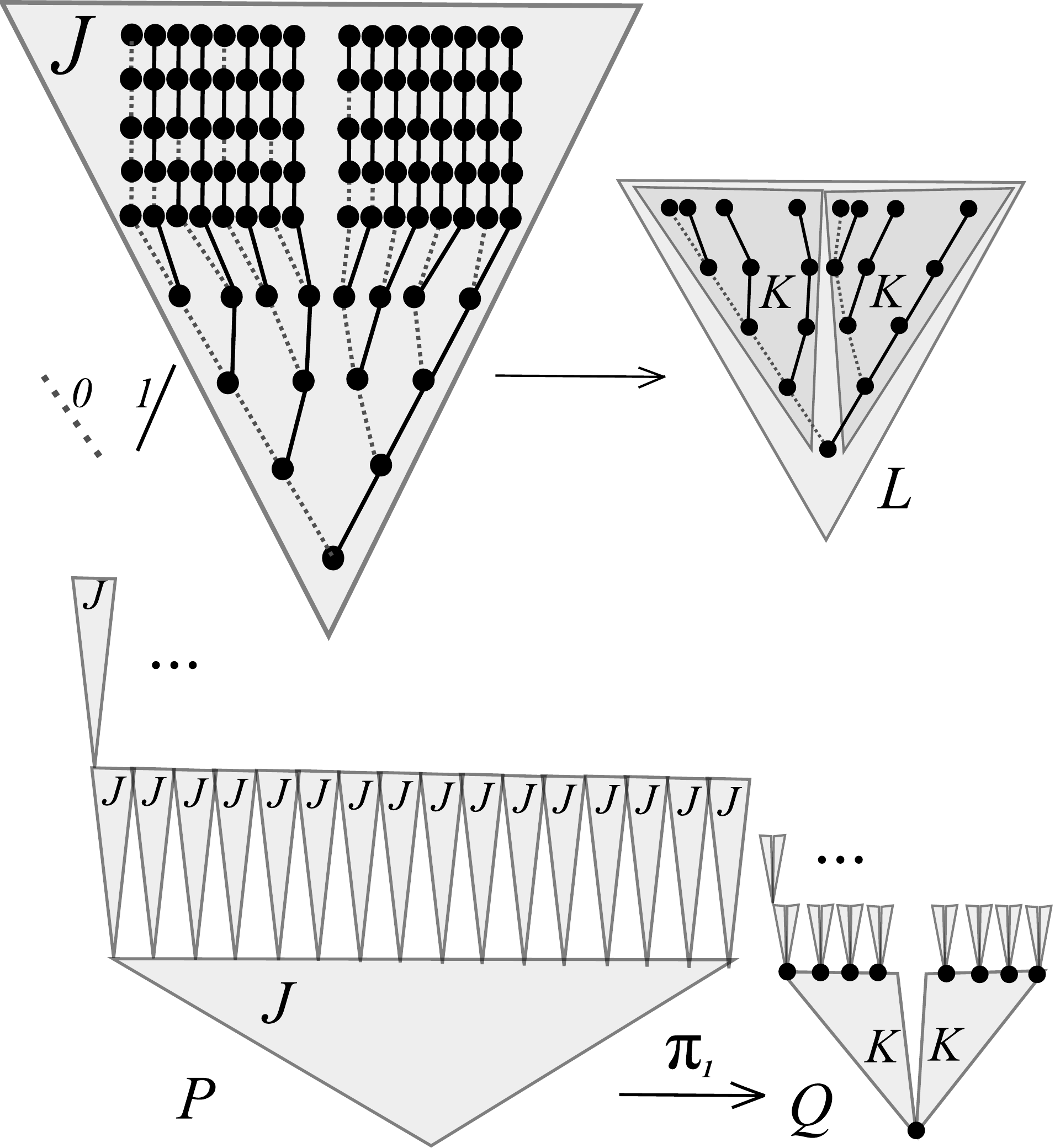} 

\caption{Proof of Theorem~\ref{upnprod}.}
\label{rys-prod-upn}
\end{figure} 

\begin{table}[h]
\centering

{\footnotesize
\begin{tabular}{c|c|c|c}
{\footnotesize $s\in J$} & {\footnotesize $w=s\left<0,2,4,6\right>$} & {\footnotesize $\mu_{Q}$} & {\footnotesize $\mu_{P}$}\\
 & {\scriptsize $\in L$} & {\scriptsize $\{x\in Q\colon x[4n,4n+3]=w\}$} & {\scriptsize $\pi_1^{-1}[\{x\in Q\colon x[4n,4n+3]=w\}]$} \\
 \hline\hline
$00000000$ & $0000$ & $1/16$ & $1/16$ \\ \hline
$00010111$ & $0001$ & $1/16$ & $1/16$\\ \hline
$00101011$ & $0111$ & &\\
$00111111$ & $0111$ & $1/4$ & $4/16$\\ 
$01101011$ & $0111$ & &\\ 
$01111111$ & $0111$ & &\\ \hline 
$01001010$ & $0011$ & $1/8$ & $2/16$\\  
$01011111$ & $0011$ & &\\ \hline
$10000101$ & $1000$ & $1/16$ & $1/16$\\ \hline
$10010111$ & $1001$ & $1/16$ & $1/16$\\ \hline
$10101111$ & $1111$ & &\\
$10111111$ & $1111$ & $1/4$& $4/16$\\
$11101111$ & $1111$ & &\\
$11111111$ & $1111$ & &\\ \hline
$11001111$ & $1011$ & $1/8$ & $2/16$\\
$11011111$ & $1011$ 
\end{tabular}}

\caption{Proof of Theorem~\ref{upnprod}.}\label{t1-upnprod}

\end{table}

Also, if $B$ is a balanced perfect set, then $\mu_{B}(Q)=0$ or $\mu_{Q}(B)=0$. The argument is the same as in the proof of Lemma~\ref{lem-bp}, namely there are two possibilities. If for all $t\in Q$, there exists infinitely many $k\in\omega$ such that $\{t\obc_{4k}\,^\frown s\colon s\in L\}\setminus T_{B}\neq\zbp$, then by Lemma \ref{lem-measure-p-0}, $\mu_{Q}(B)=0$. If it is not the case, there exists $t\in Q$ such that for all but finite $k\in\omega$, we have $\{t\obc_{4k}\,^\frown s\colon s\in L\}\subseteq T_{B}$. It follows that there exists $i\in\omega$, such that $B$ has a branching point of length $j$ for all $j\geq i$, so $s_{j+1}(B)\leq S_{j}(B)+1$, for any $j\geq i$. And since $B$ is a balanced perfect set, it implies, that $s_{j}(B)=S_j(B)$ and $s_{j+1}(B)=s_j(B)+1$ for any $j>i$. In other words, for $w\in T_{B}\cap 2^{i}$, $B\cap [w]=[w]$, and therefore for any $v\in T_B\cap 2^{4k}$ with $3k>i$, there exists $w\in 2^4$, such that $v^\frown w\in T_B\setminus T_Q$. It follows that $\mu_{B}(Q)=0$, by Lemma~\ref{lem-measure-p-0}.

Moreover, if $A$ is $Q$-null, then $A\times 2^{\omega}$ is $P$-null. Indeed, if $n\in\omega$ and $w\in L$, 
\begin{multline*}
\mu_{Q}\left(\{x\in Q\colon x[4n,4n+3]=w\}\right)=\frac{|\{s\in J\colon w=s\left<0,2,4,6\right>\}|}{16}=\\\mu_{P}\left(\pi_1^{-1}\left[\{x\in Q\colon x[4n,4n+3]=w\}\right]\right),
\end{multline*}
(see Table~\ref{t1-upnprod}). Therefore, if $\varepsilon>0$ and $\left<w_i\right>_{i\in\omega}$ is a sequence such that $w_{i}\in T_Q$, $\bigcup_{i\in\omega}[w_i]_Q$ covers $A$ and $\sum_{i\in\omega}\mu_Q\left([w_i]_Q\right)\leq \varepsilon$, then $\mu_P\left(\pi_1^{-1}\left[[w_i]_Q\right]\right)=\mu_Q\left([w_i]_Q\right)$, so $\bigcup_{i\in \omega} \pi_1^{-1}\left[[w_i]_Q\right]$ is a covering of $A\times 2^{\omega}$ of measure $\mu_P$ not greater than $\varepsilon$. 

Let $S\subseteq P$ be a Sierpiński set with respect to $\mu_{P}$, and let $X=\pi_{1}[S]\subseteq Q$. Suppose that $B$ is a balanced perfect set. Then either $\mu_{B}(Q)=0$, so $\mu_{B}(X)=0$, or $\mu_{Q}(B)=0$, so $\mu_{P}(\pi^{-1}_{1}[Q\cap B])=0$. In the latter case, $S\cap \pi^{-1}_{1}[Q\cap B]$ is countable, so $X\cap B$ is countable and $\mu_{B}(X)=0$. Hence $X\in \bPN$. 

Notice also that $\pi_2[P]=Q$ as well (see Table~\ref{t2-upnprod}). So analogously, one can check that $Y=\pi_{2}[S]\in\bPN$.

\begin{table}[h]
\centering

{\footnotesize
\begin{tabular}{c|c}
{\footnotesize $s\in J$} & {\footnotesize $w=s\left<1,3,5,7\right>\in L$} \\
 \hline\hline
$00000000$ & $0000$ \\ \hline
$00010111$ & $0111$ \\
$00111111$ & $0111$ \\ 
$10010111$ & $0111$ \\
$10111111$ & $0111$ \\ \hline
$00101011$ & $0001$ \\ \hline
$01101011$ & $1001$ \\ \hline
$01111111$ & $1111$ \\
$01011111$ & $1111$ \\
$11011111$ & $1111$  \\
$11111111$ & $1111$ \\ \hline
$01001010$ & $1000$ \\ \hline
$10000101$ & $0011$ \\
$10101111$ & $0011$ \\ \hline
$11101111$ & $1011$ \\
$11001111$ & $1011$
\end{tabular}}

\caption{$\pi_2[P]=Q$.}\label{t2-upnprod}

\end{table}

But $S\subseteq X\times Y$, so $X\times Y$ is not $P$-null, and therefore $X\times Y\notin \uPN$. \end{proof}

\begin{xrem}
The above result seems to be interesting as it resembles the argument of Recław (see \cite{ir:ppms}) which proves that if there exists a Lusin set, then the class of perfectly meager sets is not closed under taking products. Recław in his proof actually constructs a perfect set $D\subseteq 2^{\omega}\times 2^{\omega}$ and shows that given a Lusin set $L\subseteq D$, its projections are perfectly meager. The same happens in the above proof where we consider a Sierpiński set and the class $\bPN$. Nevertheless, we still do not know whether it can be done in the case of the class $\PN$.
\end{xrem}

\section{Perfectly null sets in the transitive sense}

\subsection{The definition}

In relation to the algebraic sum of sets belonging to different classes of small subsets of $2^{\omega}$, the class of perfectly sets in the transitive sense ($\PMp$) has been defined in \cite{anmstw:assrnsmzs}. The definition was also motivated by the obvious fact that a set $X$ is perfectly meager if and only if for any perfect set $P$, there exists an $F_{\sigma}$ set $F\supseteq X$ such that $F\cap P$ is meager in $P$. We  say that a set $X$ \emph{perfectly meager in the transitive sense} if for any perfect set $P$, there exists an $F_{\sigma}$ set $F\supseteq X$ such that for any $t$, the set $(F+t)\cap P$ is a meager set relative to $P$. Further properties of $\PMp$ sets were investigated in \cite{an:ratvpms}, \cite{antw:assssspmsr}, \cite{antw:smstuci} and \cite{antw:neqsipmits}, but still there are some open questions related to the properties of this class.

Obviously, a set is perfectly null if and only if for any perfect set $P$, there exists a $G_{\delta}$ set $G\supseteq X$ such that $\mu_{P}(G)=0$. We define the following new class of small sets.

\begin{defin}
We call a set $X$ \emph{perfectly null in the transitive sense} if for any perfect set $P$, there exists a $G_{\delta}$ set $G\supseteq X$ such that for any $t$, the set $(G+t)\cap P$ is $P$-null. The class of sets which are perfectly null in the transitive sense will be denoted by $\PNp$.
\end{defin}

 We do not know whether this class of sets forms a $\sigma$-ideal. 
 
 Similarly we define ideals: $\bPNp,\uPNp$ and $\vPNp$.

\begin{prop} The following sequence of inclusions holds:

\begin{tikzpicture}[description/.style={fill=white,inner sep=2pt}]
\matrix (m) [matrix of math nodes, row sep=3em,
column sep=2em, text height=1.5ex, text depth=0.25ex]
{ \PNp &\subseteq & \bPNp &\subsetneq & \uPNp&\subsetneq & \vPNp \\
\PN & \subseteq &\bPN & \subsetneq &\uPN&\subsetneq & \vPN \\ };
\path[-stealth]
(m-2-1) edge[draw=none] node [sloped] {$\supseteq$} (m-1-1)
(m-2-3) edge[draw=none] node [sloped] {$\supseteq$} (m-1-3)
(m-2-5) edge[draw=none] node [sloped] {$\supseteq$} (m-1-5)
(m-2-7) edge[draw=none] node [sloped] {$\supseteq$} (m-1-7);
\end{tikzpicture}

\end{prop}

\begin{proof} The above inclusions follow immediately from the definitions. The sets $B$ and $U$ defined in the proof of Proposition~\ref{bpnupn} are obviously also in $\uPNp\setminus \bPNp$ and $\vPNp\setminus \uPNp$, respectively. \end{proof}

\subsection{$\PNp$ sets and other classes of special subsets}

In \cite{an:ratvpms}, \cite{antw:assssspmsr}, \cite{antw:smstuci} and \cite{antw:neqsipmits}
the authors prove that $\SM\subseteq \PMp\subseteq \UM$, and that it is consistent with ZFC that those inclusions are proper. Therefore, we study the relation between the class $\PNp$ and the classes of strongly null sets and universally null sets.

\begin{thm}
Every strongly null set is perfectly null in the transitive sense.
\end{thm}

\begin{proof} Let $X$ be a strongly null set, and let $P$ be a perfect set. If $w\in T_{P}$ and $|w|=S_{n}(P)+1$, then $\mu_{P}([w]_{P})\leq 1/2^{n+1}$. It is a well-known fact that if a set $A$ is strongly null, we can obtain a sequence of open sets of any given sequence of diamiters, the union of which covers $X$ in such a way that every point of $A$ is covered by infinitely many sets from this sequence (see, e.g. \cite{lb:srl}). Therefore, let $\left<A_{n}\colon n\in\omega\right>$ be a sequence of open sets such that $X\subseteq \bigcap_{m\in\omega}\bigcup_{n\geq m} A_{n}$ and $\diam A_{n}\leq 1/2^{S_n(P)+1}$. Let $t\in 2^{\omega}$ be arbitrary. Let $B_{n}=(A_{n}+t)\cap P$. Since $\diam B_{n}\leq 1/2^{S_n(P)+1}$, we have that $B_{n}\subseteq [w_{n}]_{P}$, where $w_{n}\in T_{P}$ and $|w_{n}|=S_{n}(P)+1$. Therefore, $\mu_{P}(B_{n})\leq 1/2^{n+1}$. But $(X+t)\cap P\subseteq (\bigcap_{m\in\omega}\bigcup_{n\geq m} A_{n}+t)\cap P\subseteq \bigcap_{m\in \omega}\bigcup_{n\geq m} B_{n}$ and $\mu_{P}(\bigcap_{m\in \omega}\bigcup_{n\geq m} B_{n})=0$, so $X$ is perfectly null in the transitive sense. \end{proof} 

The following problem remains open.

\begin{prob}
Does there exist a $\PNp$ set, which is not strongly null?
\end{prob}

In particular, the authors of this paper have not been able to answer the following question.

\begin{prob}
Does there exist an uncountable $\PNp$ set in every model of ZFC?
\end{prob}

Recall that in every model of ZFC there exists an uncountable $\PMp$ set (see \cite{an:ratvpms}).

In \cite{antw:neqsipmits}, the authors prove that $\PMp\subseteq \UM$. One can ask a natural question of whether the following is true.

\begin{prob}
$\PNp\subseteq \UN$?
\end{prob}

If this inclusion holds in $ZFC$, then it is consistent with $ZFC$ that it is proper. Motivated by {\cite[Theorem 1]{ir:sapssr}}, we get the following theorem.

\begin{thm}\label{unpnts}
If there exists a universally null set of cardinality $\cont$, then there exists $Y\in\UN\setminus\bPNp\subseteq \UN\setminus \PNp$.
\end{thm}

\begin{proof} 
As in \cite{antw:neqsipmits}, we apply the ideas presented in \cite{ir:sapssr} in the case of subsets of $2^{\omega}$. Notice that there exists a perfect set $P\subseteq 2^{\omega}$ which is linearly independent over $\Z_{2}$. Indeed, define $\varp\colon 2^{<\omega}\to 2^{<\omega}$ by induction. Let $\varp(\zbp)=\zbp$. Given $\varp(w)=v\in 2^{<\omega}$ for $w\in 2^{<\omega}$ with $n=|w|$, let $\varp(w^\frown 0)=v^\frown\vare_{2k}^{2^{n+1}}$ and  $\varp(w^\frown 1)=v^\frown\vare_{2k+1}^{2^{n+1}}$, where $\vare_{l}^{m}=0\ldots 010\ldots 0$ is of length $m$ with $1$ on the $l$-th position, and $k\in\omega$ is the natural number binary notation of which is given by $w$. For example, $\varp(0)=10,$ $\varp(1)=01,$ $\varp(00)=101000,$ $\varp(01)=100100,$ $\varp(10)=010010,$ $\varp(11)=010001,$ $\varp(000)=10100010000000$, and so on. Now, notice that $\left<[\varp(w)]\right>_{w\in 2^{<\omega}}$ is a Cantor scheme, so define $P=\bigcup_{\alpha\in 2^{\omega}}\bigcap_{n\in\omega} [\varp(\alpha\obc_{n})]$. Let $\alpha_{1},\ldots,\alpha_{n}\in P$ be pairwise non-equal. There exists $l\in\omega$ such that for any $i,j\leq n, i\neq j$, $\alpha_{i}\obc_{2^{l}-2}\neq \alpha_{j}\obc_{2^{l}-2}$. Then $\alpha_{1},\ldots,\alpha_{n}$ restricted to $[2^{l}-2,2^{l+1}-2)$ are basis vectors of $2^{l}$. Thus, $P$ is linearly independent over $\Z_{2}$. The existence of such a set follows also from Kuratowski-Mycielski Theorem (see \cite[Theorem 19.1]{ak:cdst}).

Next, we follow the argument from \cite{ir:sapssr}. Let $C,D$ be perfect and disjoint subsets of $P$. We can require the set $D$ to be a balanced perfect set. Assume that $X\subseteq C$ is a universally null set and $|X|=\cont$. Let $\left<B_{x}\colon x\in X\right>$ enumerate all $G_{\delta}$ sets. For every $x\in X$, let $y_{x}\in x+D$ be such that $y_{x}\notin B_{x}$ if only $(D+x)\setminus B_{x}\neq\zbp$. Otherwise, choose any $y_{x}\in x+D$. Put $Y=\{y_{x}\colon x\in X\}$.

Notice that $+\colon C\times D\to C+D$ is a homeomorphism. Obviously, $+$ is continuous and open on $C\times D$. Since $(C+C)\cap (D+D)=\{0\}$ (because $P$ is linearly independent), we have that $+$ is one-to-one. Since $\pi_{1}\left[+^{-1}[Y]\right]=\pi_{1}\left[\{\left<x,d_{x}\right>\colon x+d_{x}=y_{x}\land x\in X\}\right]=X$ is universally null, $Y$ is universally null as well.

Now, we prove that $Y$ is not perfectly null in the transitive sense. Indeed, if $B_{x}\supseteq Y$ is a $G_{\delta}$ set, then $y_{x}\in B_{x}$, so $(D+x)\setminus B_{x}=\zbp$ and $D\cap (B_{x}+x)=D$. Therefore, $\mu_{D}(D\cap (B_{x}+x))=1$. \end{proof}

Recall that $\non(\NN)$ is the cardinal coefficient which denotes the minimal possible cardinality of a subset of the real line which is not of Lebesgue measure zero (see e.g. \cite{lb:srl}).

\begin{cor}
If $\non(\NN)=\cont$, then $\PNp\neq \UN$.
\end{cor}
\begin{proof}
If $\non(\NN)=\cont$, then there exists a universally null set of cardinality $\cont$ (see \cite[Theorem 8.8]{lb:srl}). 
\end{proof}

Taking into account Proposition~\ref{un-is-pn}, we have the following.

\begin{cor}
If $\non(\NN)=\cont$,  $\PNp\neq \PN$.
\end{cor}

\hfill $\square$

The class of perfectly meager sets in the transitive sense is closed under taking products (see \cite{antw:neqsipmits}). We do not know whether this holds for $\PNp$ sets.

\begin{prob}
Let $X,Y\in \PNp$. Is it always true that $X\times Y\in \PNp$?
\end{prob}

The answer is in the positive for $\vPNp$ sets.

\begin{prop}
Let $X,Y\in \vPNp$. Then $X\times Y\in \vPNp$.
\end{prop}

\begin{proof} Follows easily from the proof of Proposition~\ref{vpnprod}. \end{proof} 

\subsection{Additive properties of $\PNp$ sets}

We conclude this paper by investigating some additive properties of the class of sets perfectly null in the transitive sense. 

\begin{prop}\label{Avare}
Let $A\subseteq 2^{\omega}$ be open, $\mu$ be any Borel diffused measure on $2^\omega$ and $0\leq\vare<1$. Then the set $A_{\vare}=\{t\in 2^{\omega}\colon \mu(A+t)>\vare\}$ is also open.
\end{prop}

\begin{proof} Assume that $A$ is open, and let $A=\bigcup_{n\in\omega} [s_{n}]$. If $A_\vare=\zbp$, it is obviously open. Otherwise, let $t_{0}\in A_{\vare}$. There exists $N\in\omega$ such that $\mu(\bigcup_{n\leq N} [s_{n}]+t_{0})>\vare$. Let $M=\max\{|s_{n}|\colon n\leq N\}$. For any $t$ such that $t\obc_{M}=t_{0}\obc_{M}$, $\mu(A+t)\geq\mu(\bigcup_{n\leq N} [s_{n}]+t)=\mu(\bigcup_{n\leq N} [s_{n}]+t_{0})>\vare$. So $A_{\vare}$ is open.
\end{proof}

Recall that a set $A$ is called null-additive ($A\in \NN^{*}$) if for any null set $X$, $A+X$ is  null. Suppose that $\leq^*$ denotes the standard dominating order on $\omega^\omega$. 

\begin{lem}\label{srn}
Let $\mu$ be a Borel diffused measure on $2^\omega$ and $G\subseteq 2^\omega$ be a $G_{\delta}$ set. Let $Y\in\NN^{*}$ be such that for every Borel map $\varp\colon Y\to \omega^\omega$, there exists $\alpha\in \omega^{\omega}$ such that for every $y\in Y$, $\varp(y)\leq^{*}\alpha$. Moreover, assume that for all $y\in Y$, $\mu(G+y)=0$. Then $\mu(G+Y)=0$.
\end{lem}

\begin{proof}
Let $G=\bigcap_{m\in\omega}G_m$, where for any $m\in\omega$, $G_m$ is open and $G_{m+1}\subseteq G_m$. For $m\in\omega$, let $G_m=\bigcup_{i\in\omega}[w_{i,m}]$, with $w_{i,m}\in 2^{<\omega}$, $|w_{i,m}|>m$, and for $i\neq j$, $[w_{i,m}]\cap[w_{j,m}]=\zbp$. Let $F_n=\{w_{i,m}\colon i,m\in\omega\land |w_{i,m}|=n\}\subseteq 2^n$. Notice that 
\[G=\bigcap_{m\in\omega}\bigcup_{n\geq m} [F_{n}].\]
Let $\varphi\colon Y\to \omega^\omega$ be a function defined as follows:
\[\varphi(y)(k)=\min\left\{i\in\omega\colon \mu\left(\bigcup_{n\geq i}[F_n+y\obc_n]\right)\leq\frac{1}{2^{k+1}\cdot k!}\right\}\]
Notice that $\varphi$ is well defined, as $\mu(G+y)=0$ for any $y\in Y$. By Proposition \ref{Avare}, the set
\[\varphi^{-1}\left[\{\gamma\in\omega^\omega \colon \gamma(k)>i\}\right]=\left\{y\in Y\colon \mu\left(\bigcup_{n\geq i} [F_n]+y\right)>\frac{1}{2^{k+1}\cdot k!}\right\}\]
 is open for any $i,k\in\omega$, and therefore $\varphi$ is Borel, so there exists strictly increasing $\alpha\in\omega^\omega$ such that for every $y\in Y$, $\varp(y)\leq^{*}\alpha$. For $p\in\omega$, set $Y_{p}=\{y\in Y\colon \forall_{k\geq p} \varp(y)(k)\leq\alpha(k)\}$. 

Recall now the characterization of a null-additive set due to S. Shelah (see \cite[Theorem 2.7.18(3)]{tbhj:stsrl}). $A\in \NN^{*}$ if and only if for any increasing function $F\colon\omega\to\omega$, there exists a sequence $\left<I_{q} \right>_{q\in\omega}$, such that for $q\in\omega$, $I_{q}\subseteq 2^{[F(q),F(q+1))}$, $|I_{q}|\leq q$ and $A\subseteq \bigcup_{r\in\omega}\bigcap_{q\geq r} [I_{q}]$. 

Set $p\in\omega$, and apply the above characterization for $Y_{p}$ and the function $\alpha$. There exists a sequence $\left<I^p_{q} \right>_{q\in\omega}$, such that for $q\in\omega$, $I^p_{q}\subseteq 2^{[\alpha(q),\alpha(q+1))}$, $|I^p_{q}|\leq q$ and $Y_p\subseteq \bigcup_{r\in\omega}\bigcap_{q\geq r} [I^p_{q}]$. For $r\in\omega$, let $Y_{p,r}=Y\cap\bigcap_{q\geq r} [I^p_{q}]$.

Therefore, $Y_{p}=\bigcup_{r\in\omega} Y_{p,r}$. For any $q>r$, put $K_{p,q,r}=\{y\obc_{\alpha(q+1)}\colon y\in Y_{p,r}\}$. Notice that $K_{p,q,r}$ has at most 
\[\left|2^{\alpha(r)}\right|\prod_{n=r}^q\left|I^p_n\right|=2^{\alpha(r)}\prod_{n=r}^q n\leq 2^{\alpha(r)}\cdot q!\]
 elements. 

Obviously, $Y=\bigcup_{p,r\in\omega} Y_{p,r}$, so it is sufficient to prove that $\mu(G+Y_{p,r})=0$ for any $p,r\in\omega$. Notice that for $p,r\in\omega$,
\begin{multline*}
G+Y_{p,r}=\bigcup_{y\in Y_{p,r}} G+y=\bigcup_{y\in Y_{p,r}}\bigcap_{m\in\omega}\bigcup_{n\geq m}[F_n+y\obc_{n}]\subseteq \bigcap_{m\in\omega}\bigcup_{\substack{y\in Y_{p,r}\\ n\geq m}} [F_n+y\obc_{n}]\\
=\bigcap_{m\in\omega}\,\bigcup_{\substack{y\in Y_{p,r}\\ q\geq m}}\,\bigcup_{\alpha(q)\leq n<\alpha(q+1)} [F_{n}+y\obc_{n}]\subseteq \bigcap_{m\geq p}\,\bigcup_{q\geq m}\,\bigcup_{\substack{\alpha(q)\leq n<\alpha(q+1)\\w\in K_{p,q,r}}} [F_{n}+w\obc_n].
\end{multline*}
Recall that if $w\in K_{p,q,r}$, then $w=y\obc_{\alpha(q+1)}$ for some $y\in Y_{p,r}\subseteq Y_p$, thus for any $k\geq p$, $\alpha(k)\geq\varp(y)(k)$, so $\mu(\bigcup_{n\geq \alpha(k)}[F_n+y\obc_n])\leq 1/(2^{k+1}\cdot k!)$. In particular, $\mu(\bigcup_{n\geq \alpha(q)}[F_n+y\obc_n])\leq 1/(2^{q+1}\cdot q!)$, so 
\[
\mu\left(\bigcup_{q\geq m}\,\bigcup_{\substack{\alpha(q)\leq n<\alpha(q+1)\\w\in K_{p,q,r}}} [F_{n}+w\obc_n]\right)\leq 2^{\alpha(r)}\cdot\sum_{q\geq m}\frac{q!}{2^{q+1}q!}=\frac{2^{\alpha(r)}}{2^m}.
\]
Therefore, $\mu(G+Y_{p,r})\leq 2^{\alpha(r)}/2^m$ for any $m\in\omega$, so $\mu(G+Y_{p,r})=0$, for any $p,r\in\omega$.
\end{proof}

We say that a set $Y$ is $SR^{\NN}$ (see \cite{tbhj:bisr}) if for every Borel set $H\subseteq 2^{\omega}\times 2^{\omega}$ such that $H_{x}=\{y\in 2^{\omega}\colon\left<x,y\right>\in H\}$ is null for any $x\in 2^{\omega}$, $\bigcup_{x\in Y}H_{x}$ is null as well.

\begin{thm}
Let $X\in\PNp$, and let $Y$ be an $SR^{\NN}$ set. Then $X+Y\in\PN$.
\end{thm}

\begin{proof}
This theorem is an easy consequence of Lemma \ref{srn}. Indeed, by \cite[Theorem 3.8]{tbhj:bisr} if $Y$ is an $SR^{\NN}$ set, then $Y\in \NN^{*}$ and every Borel image of $Y$ into $\omega^{\omega}$ is bounded. Let $P$ be perfect. Apply Lemma \ref{srn} to measure $\mu_P$, the set $Y$ and a $G_{\delta}$ set $G$ such that $X\subseteq G$, and for all $t\in 2^{\omega}$, $\mu_P(G+t)=0$.
\end{proof}

In \cite{anmstw:assrnsmzs}, the authors prove that $\SN + \PMp\subseteq s_{0}$. The question of whether the measure analogue is true remains open.

\begin{prob}
$\SM +\PNp\subseteq s_{0}$?
\end{prob}

\begin{xrem}
Notice that a weaker statement which says that the algebraic sum of a Sierpiński set and a $\PNp$ set is an $s_{0}$-set holds by Proposition~\ref{pn-sierp-s0}.
\end{xrem}

\section*{Acknowledgement}
The authors are very grateful to the referee for a number of helpful suggestions for improvement in the paper.

\bibliography{pn-sets}

\end{document}